\documentclass [a4paper,12pt]{amsart}
\usepackage{amsthm}
\usepackage{amsmath}
\usepackage{amssymb}
\usepackage[ansinew]{inputenc}
\usepackage{amscd}
\usepackage[ansinew]{inputenc}
\usepackage[mathscr]{euscript}
\usepackage{color}
\usepackage[backref]{hyperref}
\usepackage[all]{xy}

\newtheorem{thm}{Theorem}[section]
\newtheorem{cor}[thm]{Corollary}
\newtheorem{lem}[thm]{Lemma}
\newtheorem{prop}[thm]{Proposition}

\theoremstyle{definition}
\newtheorem{ex}[thm]{Example}

\theoremstyle{definition}
\newtheorem{defn}[thm]{Definition}

\theoremstyle{definition}
\newtheorem{rem}[thm]{Remark}

\theoremstyle{definition}

\def\C{\mathbb C}

\def\Z{\mathbb Z}

\def\icis{\textsc{icis}}
\def\dim{\operatorname{dim}}

\def\D{\operatorname{Derlog}}
\def\Der{\operatorname{Der}}
\def\rad{\operatorname{rad}}

\def\id{\operatorname{id}}
\def\syz{\operatorname{syz}}

\def\Lif{\operatorname{Lif}}
\def\Low{\operatorname{Low}}

\def\ord {\operatorname{ord}}

\def\O{\mathcal O}

\def\m{\mathbf m}

\def\V{\mathbf V}
\def\I{\mathbf I}

\def\J{\mathbf J}

\def\geq{\geqslant}
\def\leq{\leqslant}

\def\*{\color{red}\blacksquare}
\def\+{\color{green}\blacksquare}

\subjclass[$2010$ Mathematics Subject Classification]{Primary
32S05; Secondary 32B30}

\newcommand{\transv}{\mathrel{\text{\tpitchfork}}}
\makeatletter
\newcommand{\tpitchfork}{%
  \vbox{
    \baselineskip\z@skip
    \lineskip-.52ex
    \lineskiplimit\maxdimen
    \m@th
    \ialign{##\crcr\hidewidth\smash{$-$}\hidewidth\crcr$\pitchfork$\crcr}
  }%
}

\def\algt{\transv^\textnormal{\tiny{º}}_\textnormal{{\tiny alg}}}

\begin{document}

\title[Mixed Bruce-Roberts numbers]{Mixed Bruce-Roberts numbers}

\author{Carles Bivi\`a-Ausina}
\address{
Institut Universitari de Matem\`atica Pura i Aplicada,
Universitat Polit\`ecnica de Val\`encia,
Cam\'i de Vera, s/n,
46022 Val\`encia, Spain}
\email{carbivia@mat.upv.es}

\author{Maria Aparecida Soares Ruas}
\address{
Instituto de Ciências Matemáticas e de Computação,
Universidade de São Paulo,
Avda. do Trabalhador São-carlense, 400,
13560-970 São Carlos, SP, Brazil}
\email{maasruas@icmc.usp.br}


\thanks{The first author was partially supported by DGICYT Grant MTM2015-64013-P and FAPESP Grant 2014/00304-2.
The second author was partially supported by CNPq Grant 306306/2015-8 and FAPESP Grant 2014/00304-2.
}

\begin{abstract}
We extend the notion of $\mu^*$-sequence and Tjurina number of functions to the framework of Bruce-Roberts numbers, that is, to pairs formed
by the germ at $0$ of a complex analytic variety $X\subseteq \C^n$ and a finitely $\mathcal R(X)$-determined analytic function germ
$f:(\C^n,0)\to (\C,0)$. We analyze some fundamental properties of these numbers.
\end{abstract}

\maketitle

\section{Introduction}


Let $\O_n$ denote the ring of analytic function germs $(\C^n,0)\to \C$ and let $\m_n$ be the
maximal ideal of $\O_n$. If $f\in\O_n$ has an isolated singularity, then we denote by
$\mu(f)$ the Minor number of $f$. That is, $\mu(f)=\dim_\C \O_n/J(f)$, where $J(f)=\langle \frac{\partial f}{\partial x_1},
\dots, \frac{\partial f}{\partial x_n}\rangle$ is the Jacobian ideal of $f$. If $H$ is a general hyperplane through the origin in $\mathbb C^n,$ then we may speak of the Milnor number of
the restriction of $f$ to $H,$ denoted by
$\mu^{(n-1)}(f)$. More generally, in \cite[p.\,300]{Cargese} Teissier introduced the sequence $\mu^*(f)=(\mu^{(n)}(f), \dots, \mu^{(1)}(f),\mu^{(0)}(f))$, where $\mu^{(i)}(f)$ denotes the Milnor number of the
restriction of $f$ to a generic linear subspace of dimension $i$ of $\C^n$, for $i=1,\dots, n$. If
$F:(\mathbb C^n\times \mathbb C,0)\to (\C,0)$ defines a
family of hypersurfaces  with isolated singularities, $f_t^{-1}(0)$, where $f_t(x)=F(t,x)$, then Teissier proves that the constancy of the sequence $\mu^{*}(f_t)$
implies the Whitney equisingularity of the pair $(F^{-1}(0)\setminus D, D),$ where $D \subset \mathbb C$ is a small disc around $0$ in the parameter space.
 In \cite{BR} Bruce and Roberts extended the notion of Milnor number to pairs
formed by an analytic function $f\in\O_n$ and an analytic subvariety $X$ of $\C^n$ (see Definition \ref{muXf}).
This number, denoted $\mu_X(f)$, is called the {\em multiplicity of $f$ on $X$} in \cite{BR}. In  some references,
$\mu_X(f)$ is called the {\em Bruce-Roberts' Milnor number of $f$ with respect to $X$}. We refer to \cite{ART2012,Gru, NOT2011}  for recent results on the relations of $\mu_X(f)$  with other classical invariants and partial results on its role on equisingularity problems in the relative case.


Let $f\in\O_n$ and let $X$ denote the germ at $0$ of an analytic subvariety of $\C^n$. This article has several purposes. We derive some consequences of the formula for $\mu_X(f)$ obtained in \cite{NOT2011} in the case
where $X$ is a weighted homogeneous hypersurface with an isolated singularity at the origin (see Theorem \ref{claucashomog}).
In particular, in the case where $X$ is a linear hyperplane in $\C^n$, there appears a relation (see Proposition \ref{primerasumademuis}) that reminds the formula of Teissier saying that if $f\in\O_n$ has an isolated singularity at the origin, then
$\mu(f)+\mu^{(n-1)}(f)$ is equal to the Samuel multiplicity of $J(f)$ in the quotient ring $\frac{\O_n}{\langle f \rangle}$ (see \cite[p.\,322]{Cargese}).

Let us observe that this multiplicity is greater than or equal to $\tau(f)$, where $\tau(f)$ is the Tjurina number of $f$, which is defined
as the colengh of the ideal $\langle f\rangle+J(f)$. By analogy with the definition of $\mu_X(f)$, in Section \ref{Tjurina}
we introduce the Bruce-Roberts' Tjurina number of $f$ with respect to $X$, which we will denote by $\tau_X(f)$. We obtain an upper bound for the quotient $\frac{\mu_X(f)}{\tau_X(f)}$ and characterize the corresponding equality.

We also extend the notion of $\mu^*$-sequence of functions to the framework of Bruce-Roberts numbers, that is, to pairs formed
by the germ at $0$ of a complex analytic variety $X\subseteq \C^n$ and a finitely $\mathcal R(X)$-determined analytic function germ
$f:(\C^n,0)\to (\C,0)$. We analyze some of the fundamental algebraic and geometric properties of these numbers.
The analogue of Teissier's result in this setting, namely, whether or not
the constancy of $\mu_X^{*}(f_t)$ implies the Whitney equisingularity of the family of function germs with isolated singularity $f_t$ with respect to a singular
variety $X$, remains an open question.


\section{The Bruce-Roberts' Milnor number}\label{qgenerals}


Let $X$ be the germ at $0$ of an analytic subvariety of $\C^n$ (for short,
we will say that $X$ is an analytic subvariety of $(\C^n,0)$). Let $I(X)$ denote the ideal of $\O_n$
generated by the germs of $\O_n$ vanishing on $X$. 
We denote by $\Theta_X$ the $\O_n$-module of germs of vector fields of $\C^n$ at $0$ which are tangent to $X$. That is
$$
\Theta_{X}=\big\{\delta \in\O_n^n: \delta(I(X))\subseteq I(X)\big\}.
$$
This module is also usually denoted by $\D(X)$ (see for instance \cite{DamonMemoirsAMS, DamonTop}).
The elements of $\D(X)$ are also known as {\it logarithmic vector fields of $X$}.
We recall that $\Theta_X$ defines a coherent sheaf of modules in a small enough neighbourhood $U$ of $0\in\C^n$.
If $x\in U$, then we denote by $\Theta_{X,x}$ the corresponding stalk at $x$. We also define the vector space
$\Theta_X(x)=\{\delta(x): \delta\in\Theta_X\}\subseteq \C^n$.
We identify any given element $\delta=(\delta_1,\dots, \delta_n)\in\O_n^n$ with the derivation
$\delta_1\frac{\partial}{\partial x_1}+\cdots+\delta_n\frac{\partial}{\partial x_n}\in\Der(\O_n)$.

Let $R$ denote an arbitrary ring and let $M$ be an $R$-module.
Given elements $u_1,\dots, u_s\in M$, we denote by $\syz(u_1,\dots, u_s)$ the module of syzygies of $\{u_1,\dots, u_s\}$.
That is, $\syz(u_1,\dots, u_s)$ is the $R$-submodule of $R^s$ formed by those $(g_1,\dots,g_s)\in R^s$ satisfying that $g_1u_1+\cdots+g_su_s=0$.
Let $I$ be an ideal of $R$, then we say that $I$ is {\it reduced} when $I$ is equal to its own radical.

The computation of $\Theta_X$ for general classes of varieties $X$ is a hard problem (see Theorem \ref{Xicisqh}). However, we will apply the following
fact in order to compute $\Theta_X$ with {\it Singular} \cite{Singular}.


\begin{lem}\label{syzyg}
Let $h=(h_1,\dots, h_m):(\C^n,0)\to (\C^m,0)$ be an analytic map such that the ideal $\langle h_1,\dots, h_m\rangle$ is reduced.
Let $X=h^{-1}(0)$. Let $D_h$ be the set of elements of $\O_n^m$ given by the columns of the matrix
$$
\left[
  \begin{array}{ccccccccccccc}
\frac{\partial h_1}{\partial x_1}&\cdots & \frac{\partial h_1}{\partial x_n}& h_1 & \cdots & h_m & 0 & \cdots &0 & \cdots & 0 &\cdots & 0 \\
\frac{\partial h_2}{\partial x_1}&\cdots & \frac{\partial h_2}{\partial x_n}& 0 & \cdots & 0 & h_1 & \cdots & h_m &  & 0 &\cdots & 0 \\
\vdots &  & \vdots & \vdots &  & \vdots & \vdots &  & \vdots & \ddots & \vdots  & & \vdots \\
\frac{\partial h_m}{\partial x_1}&\cdots & \frac{\partial h_m}{\partial x_n}& 0 & \cdots & 0 & 0 & \cdots &0 &  & h_1 & \cdots & h_m
  \end{array}
\right]
$$
Then $\Theta_X=\pi_n(\syz(D_h))$, where $\pi_n:\O_n^{n+m^2}\to \O_n^n$ is the projection onto the first $n$ components.
\end{lem}

\begin{proof}
Let $I=\langle h_1,\dots, h_m\rangle$. Since $I$ is reduced, given an element $\delta\in\O_n^n$, we have that
$\delta=(\delta_1,\dots,\delta_n)$ belongs to $\Theta_X$ if and only if $\delta(h_i)\in I$, for all $i=1,\dots, m$, which is to say that
there exist $a_1^i,\dots, a_m^i\in\O_n$ such that
$\delta_1\frac{\partial h_i}{\partial x_1}+\cdots +\delta_n\frac{\partial h_i}{\partial x_n}=
a_1^ih_1+\cdots+a_m^ih_m$, for all $i=1,\dots, m$. This latter condition is equivalent to saying that the element of $\O_n^{n+m^2}$ given by
$(\delta_1,\dots, \delta_n, -a_1^1,\dots, -a_m^1,\dots, -a_1^m,\dots, -a_m^m)$ belongs to $D_h$. Hence the result follows.
\end{proof}

If $f\in\O_n$, then we denote by $J_X(f)$ the ideal of $\O_n$ generated by
$\{\delta(f): \delta\in \Theta_X\}$. In particular, we have the inclusion $J_X(f)\subseteq J(f)$.

\begin{defn}\label{muXf}
Let $X$ be an analytic subvariety of $(\C^n,0)$ and let $f\in\O_n$. We define
\begin{equation}\label{defnBR}
\mu_X(f)=\dim_\C\frac{\mathcal O_n}{J_X(f)}.
\end{equation}
When the colength on the right of (\ref{defnBR}) is finite, the number
$\mu_X(f)$ is called the {\it multiplicity of $f$ on $X$} in \cite{BR}. In some references, $\mu_X(f)$ is called the {\it Bruce-Roberts' Milnor number of $f$ with respect to $X$} (see for instance \cite{ART2012,Gru,NOT2011}).
\end{defn}

Let $f\in\O_n$. Let us remark that, if $J_X(f)$ has finite colength, then $J(f)$ has also finite colength and $\mu_X(f)\geq \mu(f)$, since $J_X(f)\subseteq J(f)$. We also point out that when $X=\C^n$, then $\Theta_{X}=\O_n^n$ and consequently $\mu_X(f)=\mu(f)$.
When $X=\{0\}\subseteq \C^n$, then $\Theta_X=\m_n\oplus\cdots\oplus\m_n$ and hence $J_X(f)=\m_n J(f)$.

If $X\subseteq (\C^n,0)$ is the germ at $0$ of an analytic subvariety and $U$ is a sufficiently small neighbourhood of $0\in\C^n$, then
in \cite{BR} Bruce and Roberts introduced the notion of logarithmic stratification of $U$ with respect to $X$ (see \cite[Definition 1.6]{BR}),
based on the analogous notion for analytic hypersurfaces of $\C^n$ defined by Saito in \cite{Saito}. If $\{X_\alpha\}_{\alpha\in A}$
denotes this stratification, then we shall refer to $\{X\cap X_\alpha\}_{\alpha\in A}$ as the {\it logarithmic stratification of $X$}.
Some of the fundamental properties of $\{X_\alpha\}_{\alpha\in A}$ is that each stratum $X_\alpha$ is a smooth connected immersed submanifold
of $U$ and if $x\in U$ lies in a stratum $X_\alpha$, then the tangent space $T_xX_\alpha$ to $X_\alpha$ at $x$ coincides with $\Theta_X(x)$.
The germ $X$ is said to be {\it holonomic} if for some neighbourhood $U$ of $0$ in $\C^n$ the logarithmic stratification of $U$ with respect to $X$ has only finitely many strata.

Here we recall the following result from \cite[p.\,64]{BR}.

\begin{thm}\label{nuXfinita}\cite[p.\,64]{BR}
Let $X$ be an analytic subvariety of $(\C^n,0)$ and let $f\in\O_n$. Then the following conditions are equivalent:
\begin{enumerate}
\item $\mu_X(f)$ is finite.
\item $V(J_X(f))\subseteq\{0\}$.
\item $f$ has an $\mathcal R(X)$-versal unfolding.
\item $f$ is finitely $\mathcal R(X)$-determined.
\item The restriction of $f$ to each logarithmic stratum of $X$ is a submersion except, possibly, at $0$.
\end{enumerate}
\end{thm}

\begin{ex}
Let $X=\{(x,y,z)\in\C^3: xyz=0\}$ and let $f\in \O_3$ be given by $f(x,y,z)=xy+xz+yz$, for all $(x,y,z)\in\C^3$.
We observe that $\Theta_X=\langle (x,0,0),(0,y,0),(0,0,z)\rangle$. Therefore $J_X(f)=\langle xy+xz, xy+yz, xz+yz\rangle$.
In particular $\mu_X(f)$ is not finite, whereas $f$ has an isolated singularity at the origin.
\end{ex}

If $X$ is an analytic subvariety of $(\C^n,0)$, then we say that {\it $X$ supports a germ with an isolated critical point}
when there exist a germ $f\in\O_n$ such that $\mu_X(f)<\infty$. In this case we also say that {\it $f$ has an isolated singularity
on $X$ at $0$}. As shown in \cite[Theorem 3.3]{BR}, if $U$ is a sufficiently small
neighbourhood of $0\in\C^n$, then the germ $(X,x)$ supports a germ with an isolated critical point for each $x\in X\cap U$ if and only if $(X,0)$
is holonomic.

We recall that a germ of hypersurface $X\subseteq \C^n$ is said to be a {\it free divisor} when $\Theta_X$
is a free $\O_n$-submodule of $\O_n^n$ (see \cite{DamonMemoirsAMS, Saito}). In this case, necessarily $\Theta_X$ is generated by $n$ elements.

Let $g=(g_1,\dots, g_p):(\C^n,0)\to (\C^p,0)$ be an analytic map
germ. If $p\leq n$, then we denote by $\mathbf J(g_1,\dots,g_p)$ the ideal of $\O_n$ generated by the minors of order $p$ of
the Jacobian matrix of $g$. We recall that the map $g$, or the set $g^{-1}(0)$, is said to be an {\it isolated complete intersection
singularity} (or an \textsc{icis}, for short) when $p\leq n$, $\dim \V(g_1,\dots, g_p)=n-p$ and
the ideal $\langle g_1,\dots, g_p
\rangle+\J(g_1,\dots, g_p)$ has finite colength in $\O_n$. As recalled in Theorem \ref{Xicisqh}, an explicit generating system for $\Theta_X$ is known when
$X=g^{-1}(0)$, being $g:(\C^n,0)\to (\C^p,0)$ a weighted homogeneous \textsc{icis}.

If $g:(\C^n,0)\to (\C^p,0)$ is an \textsc{icis}, then we denote by $\mu(g)$ the Milnor number of $g$ (see \cite{Greuel,Le, Looijenga}).
We recall that, when $p=n$, then
\begin{equation}\label{casequidim}
\mu(g)=\dim_\C\frac{\O_n}{\langle g_1,\dots, g_n\rangle}-1
\end{equation}
(see for instance \cite[p.\,78]{Looijenga}).

Given a vector of weights $w=(w_1,\dots, w_n)\in\Z^n_{\geq 1}$, if coordinates $x_1,\dots, x_n$ in $\C^n$ are fixed,
then we define the {\it Euler vector field associated
to $w$} as $\theta_w=w_1x_1\frac{\partial}{\partial x_1}+\cdots +w_nx_n\frac{\partial}{\partial x_n}$.

As pointed out in \cite[p.\,316]{HM2}, the following result is due to Aleksandrov and Kersken
(see also \cite[p.\,467]{Alek}, \cite[p.\,79, Proposition 7.2]{BR}, \cite[p.\,617]{Wahl}).

\begin{thm}\label{Xicisqh}
Let $w\in\Z^n_{\geq 1}$ and let
$h=(h_1,\dots, h_p):(\C^n,0)\to (\C^p,0)$ be a weighted homogeneous \textsc{icis}
with respect to $w$, $n-p\geq 1$. Let $X=h^{-1}(0)$. Then $\Theta_X$ is generated by
$\{\theta_w,h_i\frac{\partial}{\partial x_j}: i=1,\dots, p,\,j=1,\dots, n\}$ and the derivations coming from
\begin{equation}\label{Ipmes1}
\I_{p+1}\left[
            \begin{array}{ccc}
              \frac{\partial}{\partial x_1} & \cdots & \frac{\partial}{\partial x_n} \\
              \frac{\partial h_1}{\partial x_1} & \cdots & \frac{\partial h_1}{\partial x_n} \\
              \vdots & \, & \vdots \\
              \frac{\partial h_p}{\partial x_1} & \cdots & \frac{\partial h_p}{\partial x_n} \\
            \end{array}
          \right].
\end{equation}
In particular, given any function $f\in\O_n$, we have
\begin{equation}\label{JXf}
J_X(f)=\langle \theta_w(f)\rangle + \langle h_1,\dots, h_p\rangle J(f)+\J(f,h_1,\dots,h_p).
\end{equation}
\end{thm}

We recall that, whenever $h:(\C^n,0)\to (\C^p,0)$ is an \textsc{icis} with $n-p\geq 1$, then the ideal $\langle h_1,\dots, h_p\rangle$ is reduced
(see \cite[p.\,7]{Looijenga}).

The case $p=1$ of Theorem \ref{Xicisqh} leads to a substantial simplification of $\Theta_X$, as can be seen in \cite[Proposition 1.2]{Wahl}.
We recall this case in the following theorem (see also \cite[p.\,249]{HM} or \cite[Theorem 2.3]{NOT2011}).

\begin{thm}\label{Derloghom}
Let $w\in\Z^n_{\geq 1}$ and let
$h\in\O_n$ such that $h$ is weighted homogeneous with respect to $w$
and $h$ has an isolated singularity at the origin, $n\geq 2$. Let $X=h^{-1}(0)$. Then
$\Theta_X$ is generated by $\theta_w$ and the derivations
$\theta_{ij}=\frac{\partial h}{\partial x_j}\frac{\partial }{\partial x_i}-\frac{\partial h}{\partial x_i}\frac{\partial}{\partial x_j}$,
for $1\leq i<j\leq n$. Hence, for all $f\in\O_n$, we have
$$
J_X(f)=\langle \theta_w(f)\rangle+ \J(f,h),
$$
for all $f\in\O_n$.
\end{thm}

\begin{rem}
Let us observe that, even if $X$ is a homogeneous \textsc{icis}, a simplification of $\Theta_X$ as in Theorem \ref{Derloghom} is not possible in general.
For instance,
let $h:(\C^3,0)\to (\C^2,0)$ be the map given by $h(x,y,z)=(x^2+y^2+z^2, xyz)$, for all $(x,y,z)\in\C^3$, and let $h^{-1}(0)$. Then,
using {\it Singular} \cite{Singular} and Lemma \ref{syzyg}, it is easy to check that the eight generators
of $\Theta_X$ given by Theorem \ref{Xicisqh} consitute a minimal generating set of $\Theta_X$.
\end{rem}



Given an analytic map germ $h=(h_1,\dots, h_p):(\C^n,0)\to (\C^p,0)$ and a function $f\in\O_n$, let us define
$$
c(f,h)=\dim_\C\frac{\O_n}{\langle h_1,\dots, h_p\rangle+ \J(f,h_1,\dots, h_p)}.
$$
Let us recall that, by \cite[Theorem 3.7.1]{Le}, if the maps $(h_1,\dots, h_p)$ and $(h_1,\dots, h_p,f)$ are \textsc{icis}, then
$c(f,h)<\infty$ and $\mu(h_1,\dots, h_p)+\mu(h_1,\dots, h_p,f)=c(f,h)$.

\begin{prop}\label{cfh}
Let $h=(h_1,\dots, h_p):(\C^n,0)\to (\C^p,0)$ be an \textsc{icis}, where $p\leq n-1$, and let $f\in\O_n$.
Let $X=h^{-1}(0)$. If $\mu_X(f)<\infty$, then $c(f,h)<\infty$.
\end{prop}

\begin{proof}
Let $I=\langle h_1,\dots, h_p\rangle+\J(f, h_1,\dots, h_p)$ and let us suppose that $\dim V(I)\geq 1$.
Let us fix a point $x\in V(I)$, $x\neq 0$. In particular $x\in V(h_1,\dots, h_p)$.
Since $h$ is an \textsc{icis}, we can assume that not all the $p\times p$ minors of the differential matrix $Dh$ vanish at $x$.
Moreover, the condition $x\in V(I)$ also implies that all $(p+1)\times (p+1)$ minors of $D(f,h)$ vanish at $x$. In particular
$\nabla f(x)$ is a linear combination of $\nabla h_1(x),\dots,\nabla h_p(x)$.

As indicated in Theorem \ref{nuXfinita}, the condition $\mu_X(f)<\infty$ implies that
the restriction of $f$ to each logarithmic stratum of $X$ is a submersion except possibly at $0$.
Let $Y$ denote the logarithmic stratum of $X$ such that $x\in Y$.
Hence, there exists some non-zero $\xi\in \Theta_{X,x}$
such that $\xi(x)$ belongs to $T_xY$ and $D(f\vert_Y)_x(\xi(x))=(Df)_x(\xi(x))\neq 0$. However, since $\nabla f(x)$ is a linear combination of $\nabla h_1(x),\dots, \nabla h_p(x)$
and $Y\subseteq V(h_1,\dots, h_p)$, it follows that $(Df)_x(\xi(x))=\nabla f(x) \cdot \xi(x)=0$, which is a contradiction. Therefore $\dim V(I)=0$,
that is, $c(f,h)<\infty$.
\end{proof}

Under the conditions of the previous result, the map $(h_1,\dots, h_p,f)$ is also an \textsc{icis} and $\mu(h_1,\dots, h_p)+\mu(h_1,\dots, h_p,f)=c(f,h)$,
by the Lê-Greuel- formula.

\begin{thm}\label{milnorfh}\cite[Proposition\,7.7,\,p.\,82]{BR}
Let $w\in\Z^n_{\geq 1}$ and let
$h=(h_1,\dots, h_p):(\C^n,0)\to (\C^p,0)$ be a weighted homogeneous \textsc{icis}
with respect to $w$, $n-p\geq 1$.
Let $f\in\O_n$ such that $\mu_X(f)<\infty$.
Then the map $(f,h_1,\dots, h_p)$ is also an \textsc{icis} and its Milnor number is given by
\begin{equation}\label{qraro}
\mu(f,h_1,\dots, h_p)=\dim_\C\frac{\O_n}{\langle \theta_w(f),h_1,\dots, h_p\rangle+\J(f,h_1,\dots, h_p)}.
\end{equation}
\end{thm}

\begin{rem}\label{sobrecfh}
Let us observe that in the proof of the above result (see \cite[p.\,83]{BR}), the application of \cite[Corollary 7.9]{BR} plays
a fundamental role. In this proof it is essential to assume that $c(f,h)<\infty$.
The original statement of \cite[Proposition\,7.7,\,p.\,82]{BR} only requires
the germ $f$ to have an isolated critical point, but actually the correct hypothesis is to assume that $\mu_X(f)<\infty$, which in turn implies
the condition $c(f,h)<\infty$, by Proposition \ref{cfh}.
\end{rem}

As a direct application of Theorem \ref{milnorfh} we have the following result, which maybe is already known for the specialists by means
of other type of techniques. 

\begin{cor}
Let $f:(\C^n,0)\to (\C,0)$ be an analytic function germ with an isolated singularity at the origin, $n\geq 2$.
Let $i\in\{1,\dots, n-1\}$. If $h_1,\dots, h_{n-i}$ denotes a family of generic linear forms of $\C[x_1,\dots, x_n]$,
then
$$
\mu^{(i)}(f)=\mu(f, h_1,\dots, h_{n-i})=\dim_\C\frac{\O_n}{\langle \theta(f),h_1,\dots, h_{n-i}\rangle+\J(f,h_1,\dots, h_{n-i})}
$$
where $\theta(f)=x_1\frac{\partial f}{\partial x_1}+ \cdots+x_n\frac{\partial f}{\partial x_n}$.
\end{cor}

\begin{proof}
It is known, by the definition of Milnor number of an \textsc{icis}, that for generic linear forms $h_1,\dots, h_{n-i}\in\C[x_1,\dots, x_n]$, we have
$\mu^{(i)}(f)=\mu(f,h_1,\dots, h_{n-i})$. Let us fix such a family of linear forms $h_1,\dots, h_{n-i}$ and let
$H=\V(h_1,\dots, h_{n-i})$. Let us remark that $(h_1,\dots, h_{n-i}):(\C^n,0)\to (\C^{n-i},0)$ is a homogeneous \textsc{icis} of dimension $i$.


By Proposition \ref{nuXfinita}, $\mu_H(f)<\infty$ if and only if
the restriction $f\vert_H$ has an isolated singularity at the origin, which is the case by taking the forms $h_1,\dots, h_{n-i}$ accordingly.
Thus the result follows as a direct application of Theorem \ref{milnorfh}.
\end{proof}

Because of its similitude with \eqref{qraro}, it is worth to recall the following result of Briançon-Maynadier \cite{BM}.

\begin{thm}\cite{BM}
Let $h:(\C^n,0)\to (\C^p,0)$ be semi-weighted homogeneous \textsc{icis} with respect to $w$.
Then $\mu(h)$ only depends on $w$ and $d_w(h)$.  Moreover $\mu(h)$ is expressed as
$$
\mu(h)=\dim_\C\frac{\O_n}{\langle \theta_w(h_1),\dots, \theta_w(h_p)\rangle+\J(h_1,\dots, h_p)}.
$$
\end{thm}

We remark that the previous result was proven by Greuel in \cite[Korollar 5.8]{Greuel} (see also \cite[(5.11.a)]{Looijenga}) when the map $h$ is assumed to be
weighted homogeneous (in this case we have $\langle \theta_w(h_1),\dots, \theta_w(h_p)\rangle=\langle h_1,\dots, h_p\rangle$).

The following theorem follows as an application of Theorem \ref{Derloghom}, Theorem \ref{milnorfh} and \cite[Corollary 7.9]{BR},
where this last result from \cite{BR} provides a formula expressing the colength of an ideal of maximal minors of a matrix as a sum of colengths of suitable ideals.

\begin{thm}\label{claucashomog}\cite[Theorem 3.1]{NOT2011}
Let $w\in\Z^n_{\geq 1}$,  $n\geq 2$. Let $h\in \C[x_1,\dots, x_n]$ be weighted homogeneous with respect to $w$
with isolated singularity at the origin and let $X=h^{-1}(0)$.
Let $f\in\O_n$ such that $\mu_X(f)<\infty$.
Then $(f,h):(\C^n,0)\to (\C^2,0)$ is an \textsc{icis} whose Milnor number satisfies the relation
\begin{equation}\label{nus1}
\mu_X(f)=\mu(f)+\mu(f,h).
\end{equation}
\end{thm}

\begin{rem}
We observe that in Theorem \ref{claucashomog} the condition that $h$ has an isolated singularity at the origin can not be removed, as
Example \ref{hnosing} shows. Obviously, if $X=h^{-1}(0)$, where $h:(\C^n,0)\to (\C,0)$ is weighted homogeneous with respect to
a given $w\in\Z^n_{\geq 1}$, and $f\in\O_n$ verifies that $\langle \theta_w(f)\rangle+\J(f,h)$ has finite colength,
then this colength is an upper bound for $\mu_X(f)$ (this bound is not tight, as is also reflected in Example \ref{hnosing}).
\end{rem}

\begin{ex}\label{hnosing}
Let $f$ and $h$ be the functions of $\O_3$ defined by $f(x,y,z)= x^3+y^3+z^3$ and $h(x,y,z)=xyz$, for all $(x,y,z)\in\C^3$.
Let $X=h^{-1}(0)$. We have that $\Theta_X=\langle (x,0,0),(0, y, 0), (0, 0,z)\rangle$. Thus $J_X(f)=\langle x^3, y^3, z^3\rangle$,
which implies that $\mu_X(f)=27$. It is straightforward to check that the ideal $\langle f,h\rangle+ \J(f,h)$ has finite colength. Hence,
$(f,h):(\C^3,0)\to(\C^2,0)$ is an $\icis$. By the Lê-Greuel formula we have the relation $\mu(f)+\mu(f,h)=
 \dim_\C\O_n/(\langle f\rangle+\J(f,h))=57$, which is different
from $\mu_X(f)$ in this case.
\end{ex}

As a direct application of Theorem \ref{claucashomog}, the following result follows.

\begin{cor}\label{muX-muY}
Let $f,h\in \C[x_1,\dots, x_n]$ be weighted homogeneous polynomials, not necessarily with respect to the same vector of weights, $n\geq 2$.
Let $X=h^{-1}(0)$ and $Y=f^{-1}(0)$.
Let us suppose that $\mu_{X}(f)$ and $\mu_{Y}(h)$ are finite. Then
$$
\mu_X(f)-\mu_Y(h)=\mu(f)-\mu(h).
$$
\end{cor}

\begin{proof}
The condition $\mu_{X}(f)<\infty$ implies that $J(f)$ has finite colength. Analogously, $J(h)$ has finite colength. Therefore, by
Theorem \ref{claucashomog}, $(f,h)$ is an \textsc{icis} and $\mu_X(f)-\mu(f)=\mu(f,h)=\mu_Y(h)-\mu(h)$.
\end{proof}

\begin{cor}\label{clau}
Let $w\in\Z^n_{\geq 1}$, $n\geq 2$. Let $h\in \C[x_1,\dots, x_n]$ be weighted homogeneous with respect to $w$
with isolated singularity at the origin. Let $f\in\O_n$. Let us suppose that the ideal $\langle\theta_w(f)\rangle+\J(f,h)$ has finite colength.
Then $\langle f\rangle+\J(f,h)$ has also finite colength
and
\begin{equation}\label{eqpp}
\dim_\C\frac{\O_n}{\langle f\rangle+\J(f,h)}=\dim_\C\frac{\O_n}{\langle \theta_w(f)\rangle+\J(f,h)}.
\end{equation}
\end{cor}

\begin{proof}
Let $X=h^{-1}(0)$. Hence $J_X(f)=\langle\theta_w(f)\rangle+\J(f,h)$.
By Proposition \ref{cfh}, we have $c(f,h)<\infty$, which implies that $(f,h)$ is an \textsc{icis}.
Since $f$ has an isolated singularity at the origin, we have that
$$
\mu(f)+\mu(f,h)=\dim_\C\frac{\O_n}{\langle f\rangle+\J(f,h)},
$$
by \cite[Theorem 3.7.1]{Le}.
Then (\ref{eqpp}) follows as a direct consequence of Theorem \ref{Derloghom}.
\end{proof}

Let $f\in \O_n$ and let $i\in \{1,\dots,n\}$. By virtue of Theorem \ref{Xicisqh} and the upper semicontinuity of
the colength of ideals, we can consider the minimum value of $\mu_H(f)$ when $H$ varies in the set of
linear subspaces of $\C^n$ of dimension $i$. Let us denote this number by $\mu_{H^{(i)}}(f)$. We will also write
$f\vert_{H^{(i)}}$ to refer to the restriction of $f$ to a generic linear subspace of $\C^n$ of dimension $i$.

Let $I$ be an ideal of finite colength in a Noetherian local ring $(R,\m)$ of dimension $d$ and let $i\in\{0,1,\dots, d\}$.
Then $e_i(I)$ will denote the mixed multiplicity $e(I,\dots, I,\m,\dots,\m)$, where $I$ is repeated
$i$ times and $\m$ is repeated $n-i$ times (we refer to \cite{HS} and \cite{Cargese} for the definition and basic properties
of mixed multiplicities). We recall that 
$e_n(I)=e(I)$, where
$e(I)$ denotes the Samuel multiplicity of $I$.

\begin{prop}\label{primerasumademuis}
Let $f\in\O_n$ and let $i\in\{0,1,\dots, n-1\}$. If $f$ has an isolated singularity at the origin, then
\begin{equation}\label{sumademuis}
\mu_{H^{(i)}}(f\vert_{H^{(i+1)}})=\mu^{(i+1)}(f)+\mu^{(i)}(f)=e_i\left(J(f)\frac{\O_n}{\langle f\rangle}\right).
\end{equation}
\end{prop}

\begin{proof}
The second equality in \eqref{sumademuis}, for all $i\in\{0,1,\dots, n-1\}$, is a result of Teissier in \cite[p.\,322]{Cargese}.
Let $H$ be a linear subspace of $\C^n$ of dimension $n-1$ and let $h\in\C[x_1,\dots,x_n]$ be a linear form
such that $H=h^{-1}(0)$. Since the logarithmic stratification of $H$ is given by $H$ itself,
Theorem \ref{nuXfinita} shows that $\mu_H(f)<\infty$ if and only if the restriction of $f$ to $H$ is a submersion except, possibly, at the origin, which is
to say that the restriction $f\vert_H$ has, at most, an isolated singularity at the origin. The latter condition holds for a generic choice of
$H$ in the Grassmannian variety of linear subspaces of $\C^n$ of dimension $n-1$ (see for instance \cite[p.\,299]{Cargese}).
Therefore, we can apply Theorem \ref{claucashomog} to say that for a generic linear subspace $H$ of $\C^n$ of dimension $n-1$ we have that
$\mu_H(f)=\mu(f)+\mu(f,h)$. We recall that $\mu(f,h)=\mu(f\vert_H)=\mu^{(n-1)}(f)$. Hence
\begin{equation}\label{sumademuis2}
\mu_{H^{(n-1)}}(f)=\mu(f)+\mu^{(n-1)}(f).
\end{equation}

Let us fix an index $i\in\{1,\dots, n-1\}$. If we apply \eqref{sumademuis2} to $f\vert_{H^{(i+1)}}$, then we obtain that
$\mu_{H^{(i)}}(f\vert_{H^{(i+1)}})=\mu(f\vert_{H^{(i+1)}})+\mu^{(i)}(f\vert_{H^{(i+1)}})=\mu^{(i+1)}(f)+\mu^{(i)}(f)$.
\end{proof}

In the next example we see that the numbers $\mu_{H^{(i)}}(f\vert_{H^{(i+1)}})$ and $\mu_{H^{(i)}}(f)$
are different in general. Let us remark that in the first case the subscript $H^{(i)}$ makes reference to a linear subspace of codimension $1$ in $\C^{i+1}$.

\begin{ex}
Let $f\in\O_4$ be the function given by $f(x,y,z,t)=x^3+xy^4+y^3z+t^3+yz^5$. We have that
$\mu^*(f)=(60, 12, 4,2,1)$. Therefore relation (\ref{sumademuis}) shows that $\mu_{H^{(0)}}(f\vert_{H^{(1)}})=3$, $\mu_{H^{(1)}}(f\vert_{H^{(2)}})=6$,
$\mu_{H^{(2)}}(f\vert_{H^{(3)}})=16$. Moreover $\mu_{H^{(3)}}(f)=72$, $\mu_{H^{(2)}}(f)=68$
and $\mu_{H^{(1)}}(f)=66$, $\mu_{H^{0}}(f)=64$.
\end{ex}

The following result shows another aspect of Bruce-Roberts' Milnor numbers.

\begin{cor}
Let $f:(\C^n,0)\to (\C,0)$ be a weighted homogeneous function with an isolated singularity at the origin.
Let $Y=f^{-1}(0)$. Then
$$
\mu^{(n-1)}(f)=\mu_{Y}(h)
$$
for a generic choice of a linear form $h\in\C[x_1,\dots, x_n]$.
\end{cor}

\begin{proof}
Let $h\in\C[x_1,\dots, x_n]$ be a generic linear form. Let $X=h^{-1}(0)$. Obviously, the restriction of $h$ to any logarithmic stratum of $Y$
is a submersion except possibly at $0$. Therefore $\mu_Y(h)<\infty$.
By Corollary \ref{muX-muY} we have that $\mu_X(f)=\mu_Y(h)+\mu(f)-\mu(h)=\mu_Y(h)+\mu(f)$, since $\mu(h)=0$.
Moreover, by (\ref{sumademuis2}) we obtain that $\mu_X(f)=\mu_{H^{(n-1)}}(f)=\mu(f)+\mu^{(n-1)}(f)$. Joining both relation, the result follows.
\end{proof}

\section{The Bruce-Roberts' Tjurina number}\label{Tjurina}

In this section we introduce the notion of Tjurina number in the context described in the previous section. We will compare this number with Bruce-Roberts'
Milnor numbers in Theorem \ref{boundquot}.

\begin{defn}
Let $X$ be an analytic subvariety of $(\C^n,0)$ and let $f\in\O_n$. We define
\begin{equation}\label{defnBRT}
\tau_X(f)=\dim_\C\frac{\mathcal O_n}{\langle f \rangle+J_X(f)}.
\end{equation}
When the colength on the right of (\ref{defnBRT}) is finite, we refer to $\tau_X(f)$ as
the {\it Bruce-Roberts' Tjurina number of $f$ with respect to $X$}.
\end{defn}

Let $R$ be a ring and let $I$ be an ideal of $R$. Let $f\in R$. We denote by $r_f(I)$ the minimum of those $r\in\Z_{\geq 1}$
such that $f^r\in I$. If no such $r$ exist, then we set $r_f(I)=\infty$.
Let us also denote by $\varphi_{f,I}$ the morphism $R/I\to R/I$ defined by $g+I\mapsto fg+I$, for all $g\in R$.
If $M$ is an $R$-module, then we denote by $\ell(M)$ the length of $M$. As usual, we refer to $\ell(R/I)$ as the colength of $I$.
With aim of comparing Bruce-Robert's Milnor and Tjurina numbers, we show the following result, which is inspired by the main result of Liu in \cite{Liu}.

\begin{thm}\label{boundquot}
Let $(R,\m)$ be a Noetherian local ring.
Let $I$ be an ideal of $R$ of finite colength and let $f\in R$ such that $r_f(I)<\infty$. Then
\begin{equation}\label{lengths}
\frac{ \ell\left(\frac{R}{I}\right)}{\ell\left(\frac{R}{\langle f\rangle + I}\right)}\leq r_f(I)
\end{equation}
and equality holds if and only if $\ker(\varphi_{f,I})=\frac{\langle f^{r-1}\rangle+I}{I}$, where $r=r_f(I)$.
\end{thm}

\begin{proof}
Let $A=R/I$ and $B=R/(\langle f\rangle+I)$. Let $r=r_f(I)$.
Let us consider the following chain of ideals
\begin{equation}\label{chain}
0=\frac{\langle f^r\rangle +I}{I}\subseteq \frac{\langle f^{r-1}\rangle +I}{I} \subseteq \cdots \subseteq
\frac{\langle f^2\rangle +I}{I} \subseteq \frac{\langle f\rangle +I}{I} \subseteq A.
\end{equation}

From (\ref{chain}) it follows that
\begin{equation}\label{lengths1}
\ell\left(\frac{R}{I}\right)=
\sum_{i=0}^{r-1}\ell\left(\frac{\langle f^{i}\rangle+I}{\langle f^{i+1}\rangle+I}\right).
\end{equation}

Let $\varphi=\varphi_{f,I}$.
It is immediate to see that the sequence
\begin{equation}\label{exactabasica}
\xymatrix@C=0.2cm@R=1.5ex{
0 \ar[rr] &&  \ker(\varphi)   \ar[rr]^-{j}  &&  \displaystyle\frac{R}{I} \ar[rr]^-{\varphi}   && \displaystyle\frac{R}{I}    \ar[rr] &&
\displaystyle\frac{R}{\langle f\rangle +I}  \ar[rr] && 0.
}
\end{equation}
is exact. So
\begin{equation}\label{lengths2}
\ell\big(\ker(\varphi)\big)=\ell\left(\frac{R}{\langle f\rangle+I}\right).
\end{equation}
Let us fix any $i\in\{1,\dots, r-1\}$. The sequence \eqref{exactabasica} induces the exact sequence
\begin{equation}\label{exactabasica2}
\xymatrix@C=0.2cm@R=1.5ex{
0 \ar[rr] &&  \ker(\varphi)\cap {\displaystyle\frac{\langle f^i\rangle +I}{I}} \ar[rr]^-{j} &&
\displaystyle\frac{\langle f^i\rangle +I}{I}  \ar[rr]^{\varphi}   && \displaystyle\frac{\langle f^i\rangle +I}{I}    \ar[rr] &&
\displaystyle\frac{\langle f^i\rangle +I}{\langle f^{i+1}\rangle+I}  \ar[rr] && 0.
}
\end{equation}

The exactness of \eqref{exactabasica2} implies that
\begin{equation}\label{lengths3}
\ell\left(\ker(\varphi)\cap \frac{\langle f^i\rangle+I}{I}\right)=\ell\left(\frac{\langle f^i\rangle+I}{\langle f^{i+1}\rangle+I}\right).
\end{equation}

Relations \eqref{lengths2} and \eqref{lengths3} imply that
\begin{equation}\label{lengths4}
\ell\left(\frac{\langle f^i\rangle+I}{\langle f^{i+1}\rangle+I}\right)\leq \ell\left(\frac{R}{\langle f\rangle+I}\right)
\end{equation}
for all $i=1,\dots, r-1$. Hence, by \eqref{lengths1}, we have that
$$
\ell\left(\frac{R}{I}\right)=
\sum_{i=1}^{r-1}\ell\left(\frac{\langle f^{i}\rangle+I}{\langle f^{i+1}\rangle+I}\right)+
\ell\left(\frac{R}{\langle f\rangle+I}\right)\leq r \ell\left(\frac{R}{\langle f\rangle+I}\right)
$$
and thus \eqref{lengths} follows. The above relation shows that
\begin{align*}
\ell\left(\frac{R}{I}\right)=r \ell\left(\frac{R}{\langle f\rangle+I}\right)
&\Longleftrightarrow \ell\big(\ker(\varphi)\big)=\ell\left(\ker(\varphi)\cap \frac{\langle f^i\rangle+I}{I}\right),\,\textnormal{for all $i=1,\dots, r-1$}\\
&\Longleftrightarrow \ker(\varphi)=\ker(\varphi)\cap \frac{\langle f^i\rangle+I}{I},\,\textnormal{for all $i=1,\dots, r-1$}\\
&\Longleftrightarrow \ker(\varphi)\subseteq \frac{\langle f^{r-1}\rangle +I}{I}\Longleftrightarrow \ker(\varphi)= \frac{\langle f^{r-1}\rangle +I}{I},
\end{align*}
where the last equivalence follows as a consequence of the definition of $r$.
\end{proof}

We remark that it is easy to find examples where the analogous inequality to \eqref{lengths} obtained when replacing
the ideal $\langle f\rangle$ by an arbitrary ideal does not hold in general.
As an immediate application of the previous theorem we have the following result.

\begin{cor}\label{boundmutau}
Let $X$ be an analytic subvariety of $(\C^n,0)$. Let $f\in\O_n$ such that $\mu_X(f)<\infty$. Then
\begin{equation}\label{mutau}
\frac{\mu_X(f)}{\tau_X(f)}\leq r_{f}(J_X(f))
\end{equation}
and equality holds if and only if $\ker(\varphi_{f,J_X(f)})=\frac{\langle f^{r-1}\rangle+J_X(f)}{J_X(f)}$, where $r=r_f(J_X(f))$.
\end{cor}

\begin{cor}
Let $w\in\Z^n_{\geq 1}$ and let
$h=(h_1,\dots, h_p):(\C^n,0)\to (\C^p,0)$ be a weighted homogeneous \textsc{icis}
with respect to $w$, $p\leq n-1$.
Let $f\in\O_n$ such that $\mu_X(f)<\infty$.
Then the map $(f,h_1,\dots, h_p)$ is also an \textsc{icis} and
\begin{equation}\label{mudeh}
\mu(h)\leq (r-1)\mu(f,h)
\end{equation}
where $r=r_{\pi(\theta_w(f))}\left(\pi(J(f,h_1,\dots, h_p))\right)$ and $\pi$ denotes the natural projection $\O_n\to \frac{\O_n}{\langle h_1,\dots, h_p\rangle}$. Moreover, if $R$ denotes the ring $\O_n/\big(\langle h_1,\dots, h_p\rangle+ \J(f,h_1,\dots, h_p)\big)$, then
equality holds in \eqref{mudeh} if and only if the kernel of the
automorphism of $R$ defined by multiplication by $\theta_w(f)$ is equal to
the ideal generated by the image of $\theta_w(f)^{r-1}$ in $R$.

\end{cor}

\begin{proof}
By Theorem \ref{milnorfh}
we know that $(f,h_1,\dots, h_p)$ is an \textsc{icis} whose Milnor number is equal to the colength of the ideal
$\pi(\langle \theta_w(f)\rangle+\J(f,h_1,\dots, h_p))$ in $\frac{\O_n}{\langle h_1,\dots, h_p\rangle}$.
By Proposition \ref{cfh}
we also know that the number
$c(f,h)$ is finite. Let us recall that $c(f,h)$ is equal to the colength of $\pi(\J(f,h_1,\dots, h_p))$.
Therefore, by Theorem \ref{boundquot} and the Lê-Greuel formula, we obtain that
$$
\frac{\mu(h)+\mu(f,h)}{\mu(f,h)}=\frac{c(f,h)}{\mu(f,h)}\leq r,
$$
which is equivalent to saying that $\mu(h)\leq (r-1)\mu(f,h)$. The characterization of equality in (\ref{mudeh}) is a direct application of
Theorem \eqref{boundquot} .
\end{proof}


The bound given in \eqref{boundmutau} is sharp, as the following example shows.

\begin{ex}
Let $h\in \O_2$ be the polynomial given by $h(x,y)=xy^6+x^4y^4+x^{10}$ and let $X=h^{-1}(0)$.
Hence $\Theta_X=\langle (-2x^4y^3, 5y^6+2x^3y^4+5x^9),(2x, 3y)\rangle$.
Let us consider the function
$f(x,y)=x+y$. We have $\mu_X(f)=6$ and $\tau_X(f)=1$. Moreover $r_f(J_X(f))=6$. This shows that in this example
equality holds in \eqref{mutau}.
\end{ex}


\section{Derlog and lowerable vector fields}\label{Lowandlift}

Given an integer $i\in\{1,\dots, n\}$, we denote by $\pi_i$ the projection $\C^n\to \C^{i}$ onto the first $i$ coordinates.
Let $\id_{\C^n}$ be the identity map $\C^n\to \C^n$.
We denote by $L_{i,n}$ the set of linear maps $p:\C^i\to\C^n$ such that $\pi_i\circ p=\id_{\C^i}$, that is, of the form
$$
p(x_1,\dots, x_i)=\big(x_1,\dots, x_i, \ell_{i+1}(x_1,\dots, x_i), \dots, \ell_{n}(x_1,\dots, x_i)\big),
$$
where $\ell_{i+1},\dots,\ell_n$ denote linear forms of $\C[x_1,\dots, x_n]$.
If $1\leq i\leq n-1$, then the set $L_{i,n}$
can be identified with the set of matrices of size $(n-i)\times i$ with entries in $\C$.

Let $X\subseteq(\C^n,0)$ be an analytic subvariety, $n\geq 2$, and let $p\in L_{i,n}$, where $i\in\{1,\dots, n-1\}$.
The aim of this section is to obtain information about $\Theta_{p^{-1}(X)}$ in terms of $p$ and $\Theta_X$.

\begin{defn}\label{lowerables}
Let $p:\C^{i}\to \C^n$ be a linear map, where $i\in\{1,\dots, n\}$, and let $X$ be an analytic subvariety of $(\C^n,0)$. We define
$$
\Low_X(p)=\big\{\theta\in\O_i^i: Dp\circ\theta= \eta\circ p,\,\, \textnormal{for some $\eta\in\Theta_X$}\big\},
$$
where $Dp$ denotes the differential of $p$.
The elements of $\Low_X(p)$ are also known as {\it lowerable vector fields with respect to $p$ and $X$}.

If $\eta\in\Theta_X$ verifies that there exists some $\theta\in\O_i^i$ such that $Dp\circ\theta= \eta\circ p$, then we
say that $\eta$ is {\it liftable with respect to $p$}.
Let us denote by $\Lif_X(p)$ the set of such vector fields.
Let us remark that $\Low_X(p)$ is an $\O_i$-submodule of $\O_i^i$ and
$\Lif_X(p)$ is an $\O_n$-submodule of $\O_n^n$.
\end{defn}

Let us fix a map $p\in L_{i,n}$, for some $i\in\{1,\dots, n\}$, and let $J(p)$ denote the Jacobian module of $p$, that is,
$J(p)=\langle \frac{\partial p}{\partial x_1},\dots,\frac{\partial p}{\partial x_i} \rangle\subseteq\O_i^n$.
By abuse of notation, let us also denote by $\pi_i$ the projection $\O_n^n\to \O_n^i$ onto the first $i$ components.
Let $p^*(\Theta_X)=\{\eta\circ p: \eta\in\Theta_X\}\subseteq\O_i^n$.
An elementary computation shows that
\begin{align}
\Lif_X(p)&=\big\{\eta\in\Theta_X: p\big(\pi_i(\eta \circ p)\big)=\eta \circ p \big\}  \label{simplelift} \\
\Low_X(p)&=\big\{\pi_{i}(\eta \circ p): \eta \in\Lif_X(p)\big\}=\pi_i\big(p^*(\Theta_X)\cap J(p)\big) \label{simplelow} .
\end{align}

Given a map $p:\C^i\to \C^n$ and an analytic subvariety $X\subseteq (\C^n,0)$, then $p$ is said to be {\it algebraically transverse to $X$ off $0$}
when there exists an open neighbourhood $U$ of $0$ in $\C^n$ such that
\begin{equation}\label{transv1}
Dp(T_x\C^i)+\Theta_X(p(x))=T_{p(x)}\C^n
\end{equation}
for all $x\in U\setminus\{0\}$.
We will denote this condition by $p\algt X$. We recall that $p\algt X$ if and only if $p$ is finitely $\mathcal K_X$-determined (see \cite[p.\,9]{DamonMemoirsAMS}). Let us remark that if $p$ is an immersion, then relation (\ref{transv1}) holds only if $\dim_\C \Theta_X(p(x))\geq n-i$.
Here we recall a result from \cite{DamonMemoirsAMS} relating the modules $\Theta_{p^{-1}(X)}$ and $\Low_X(p)$.

\begin{thm}\label{LX} \cite[p.\,17]{DamonMemoirsAMS}
Let $X$ be an analytic subvariety of $(\C^n,0)$ and let $p:\C^i\to \C^n$ be a map such that $p\algt X$.
Then there exists some $k\geq 1$ such that
\begin{equation}\label{Damonincl}
\m_i^k\Theta_{p^{-1}(X)}\subseteq \Low_X(p)\subseteq \Theta_{p^{-1}(X)}.
\end{equation}
\end{thm}

The following example shows that the second inclusion of (\ref{Damonincl}) can be strict. In Proposition \ref{inclLow} we give a sufficient condition for the inclusion $\Low_X(p)\subseteq \Theta_{p^{-1}(X)}$ to hold without imposing the condition $p\algt X$.

\begin{ex}
Let us consider the function $h\in\O_2$ given by $h(x,y)=x^3y^2+x^2y^3+x^6+y^6$ and let $X=h^{-1}(0)$.
We observe that $X$ is a plane curve with an isolated singularity at the origin.
Let us consider the immersive linear map $p:\C\to \C^2$ given by $p(x)=(x,x)$, for all $x\in \C$.
Hence $h(p(x))=2x^5(1+x)$, for all $x\in\C$, which implies that $p^{-1}(X)=\{0\}$, as germs at $0$.
In particular, there exists an open neighbourhood $U$ of $0\in\C$ such that $h(p(x))\neq 0$, for all $x\in U\setminus\{0\}$. Therefore, the dimension of $\Theta_X(p(x))$ as a complex vector space is $2$, for all $x\in U\setminus\{0\}$. This shows that $p\algt X$.
A basic computation with {\it Singular} \cite{Singular} shows that $\Low_X(p)=\pi_1(p^*(\Theta_X)\cap J(p))=\langle x^3\rangle$, whereas
$\Theta_{p^{-1}(X)}=\langle x\rangle$. That is, $\Low_X(p)\subsetneq \Theta_{p^{-1}(X)}$.
\end{ex}

Let $h:(\C^n,0)\to (\C^m,0)$ be an analytic map. We say that $h$ is {\it reduced} when the ideal of $\O_n$
generated by the components of $h$ is reduced.


\begin{prop}\label{inclLow}
Let $X$ be an analytic subvariety of $(\C^n,0)$, $n\geq 2$, and let $i\in\{1,\dots, n-1\}$. Let $h:(\C^n,0)\to (\C^m,0)$ be a reduced analytic map
such that $X=h^{-1}(0)$ and let $p\in L_{i,n}$ such that the map $h\circ p:(\C^i,0)\to (\C^n,0)$ is also reduced. Then
$$
\Low_X(p)\subseteq \Theta_{p^{-1}(X)}.
$$
\end{prop}

\begin{proof}
Let $J$ be the ideal of $\O_n$ generated by the components of $h$ and let $\theta\in \Low_X(p)$, $\theta=(\theta_1,\dots, \theta_i)$. By relations (\ref{simplelift}) and (\ref{simplelow}) it follows that there exists some $\eta\in \Theta_X$ such that $p(\pi_i(\eta \circ p))=\eta \circ p$ and $\theta=\pi_i(\eta\circ p)$.

Let $p^*:\O_{n}\to \O_{i}$ be the morphism given by $p^*(f)=f\circ p$, for all $f\in\O_n$.
We have that $I(p^{-1}(X))=I((h\circ p)^{-1}(0))=
\rad(p^*(J))=p^*(J)$. We will see that $\theta(h_k\circ p)\in p^*(J)$, for all $k=1,\dots, m$, where $h=(h_1,\dots, h_m)$.


Let us write $p$ as $p(x_1,\dots, x_i)=(x_1,\dots, x_i, \sum_{j=1}^i a_{i+1,j}x_j, \dots,\sum_{j=1}^i a_{n,j}x_j)$,
for some coefficients $a_{\ell,j}\in \C$. Let us fix an index $k\in\{1,\dots, m\}$.
Computing $\theta(h_k\circ p)$ we obtain the following:
\begin{align*}
\theta(h_k\circ p)&=\sum_{j=1}^{i}\theta_j\frac{\partial (h_k\circ p)}{\partial x_j}
=\sum_{j=1}^{i}\theta_j\left(\frac{\partial h_k}{\partial x_j}\circ p+\sum_{\ell=i+1}^n a_{\ell, j}\frac{\partial h_k}{\partial x_\ell}\circ p\right)\\
&=\sum_{j=1}^{i}\theta_j\left(\frac{\partial h_k}{\partial x_j}\circ p\right)+\sum_{j=1}^i \theta_j\left(\sum_{\ell=i+1}^n a_{\ell, j}\frac{\partial h_k}{\partial x_\ell}\circ p\right)\\
&=\sum_{j=1}^{i}(\eta_j\circ p)\left(\frac{\partial h_k}{\partial x_j}\circ p\right)+
\sum_{\ell=i+1}^n \left( \sum_{j=1}^i \theta_j a_{\ell, j}\right)\frac{\partial h_k}{\partial x_\ell}\circ p\\
&=\sum_{j=1}^{n}(\eta_j\circ p)\left(\frac{\partial h_k}{\partial x_j}\circ p\right)=\eta(h_k)\circ p\in p^*(J).
\end{align*}
Therefore the inclusion $\Low_X(p)\subseteq \Theta_{p^{-1}(X)}$ holds.
\end{proof}

Let $n\in\Z_{\geq 1}$ and let us fix coordinates $x_1,\dots,x_n$ in $\C^n$.
Then we denote by $\theta^{(n)}$ the Euler derivation $x_1\frac{\partial }{\partial x_1}+\cdots+x_n\frac{\partial }{\partial x_n}$.
In the next result we show a case where the equality $\Low_X(p)=\Theta_{p^{-1}(X)}$ holds.

\begin{prop}\label{LowX}
Let $h:(\C^n,0)\to (\C^m,0)$ be a homogeneous \textsc{icis} such that $n-m\geq 1$ and let $X=h^{-1}(0)$.
Let $i\in\{m+1,\dots,n\}$ and let $p:\C^i\to\C^n$ be an immersive linear map such that $h\circ p:(\C^i,0)\to (\C^m,0)$ is an
\textsc{icis} of positive dimension. Then
\begin{equation}\label{LowandDer}
\Low_X(p)=\Theta_{p^{-1}(X)}.
\end{equation}
\end{prop}

\begin{proof}

Let $H$ denote the image of $p$.
Let $R:\C^n\to\C^n$ be a rotation such that $R(H)$ is given by the equations $x_{i+1}=\cdots=x_n=0$.
Let $q=R\circ p:\C^i\to \C^n$. Therefore $q(x_1,\dots, x_i)=(x_1,\dots, x_i,0,\dots, 0)$, for all $(x_1,\dots, x_i)\in \C^i$.
Let $Z=R(X)$. 

Let $Y=q^{-1}(Z)=p^{-1}(X)$. By hypothesis, $Y$ is a homogeneous \textsc{icis}. Let $f=h\circ R^{-1}$. Therefore $Z=f^{-1}(0)$ and $Y=(f\circ q)^{-1}(0)$.
Let us write $f=(f_1,\dots, f_m):(\C^i,0)\to (\C^m
,0)$.

Let us consider the matrices
$$
A_Y=\left[
            \begin{array}{ccc}
              \frac{\partial}{\partial x_1} & \cdots & \frac{\partial}{\partial x_i} \\
              \frac{\partial(f_1\circ q)}{\partial x_1} & \cdots & \frac{\partial(f_1\circ q)}{\partial x_i} \\
              \vdots & \, & \vdots \\
              \frac{\partial(f_m\circ q)}{\partial x_1} & \cdots & \frac{\partial(f_m\circ q)}{\partial x_i} \\
            \end{array}
          \right], \hspace{1cm}
A_Z=\left[
            \begin{array}{ccc}
              \frac{\partial}{\partial x_1} & \cdots & \frac{\partial}{\partial x_n} \\
              \frac{\partial f_1}{\partial x_1} & \cdots & \frac{\partial f_1}{\partial x_n} \\
              \vdots & \, & \vdots \\
              \frac{\partial f_m}{\partial x_1} & \cdots & \frac{\partial f_m}{\partial x_n} \\
            \end{array}
          \right].
$$
By Theorem \ref{Xicisqh},
we have that $\Theta_Y$ is generated by $\{\theta^{(i)},(f_\ell\circ q)\frac{\partial}{\partial x_j}: \ell=1,\dots, m,\,j=1,\dots, i\}$ and the derivations coming from $\I_{m+1}(A_Y)$. Let us denote this generating system by $W_Y$.
Also by Theorem \ref{Xicisqh}, a generating system of $\Theta_Z$ is given by
$\{\theta^{(n)},f_\ell\frac{\partial}{\partial x_j}: \ell=1,\dots, m,\,j=1,\dots, n\}$ and the derivations coming from $\I_{m+1}(A_Z)$.
Let us denote this generating system of $\Theta_Z$ by $W_Z$.
Given indices $1\leq j_1<\cdots <j_{m+1}\leq i$, let $\theta_{j_,\dots, j_{m+1}}$ denote
the minor of $A_Y$ formed by the columns $j_1,\dots, j_{m+1}$ of $A_Y$ and let $\theta'_{j_1,\dots, j_{m+1}}$ denote the analogous minor of $A_Z$.
Then, it is immediate to check that the following relations hold:
\begin{align*}
\theta^{(i)}&=\pi_i(\theta^{(n)}\circ q)\\
(f_\ell\circ q)\frac{\partial}{\partial x_j}&=\pi_i\big((f_\ell\frac{\partial}{\partial x_j})\circ q\big),\,\, \textnormal{for all $\ell=1,\dots, m$, $j=1,\dots, i$}\\
\theta_{j_,\dots, j_{m+1}}&=\pi_i(\theta'_{j_,\dots, j_{m+1}}\circ q),\,\,\textnormal{for all $1\leq j_1<\cdots <j_{m+1}\leq i$}.
\end{align*}



Therefore we found that for any $\theta\in W_Y$, there exists some $\eta=(\eta_1,\dots,\eta_n)\in W_Z$ such that
$\theta=\pi_i(\eta\circ q)$ and $\eta_{i+1}=\cdots=\eta_n=0$. In particular
$\eta_{i+1}\circ q=\cdots=\eta_n\circ q=0$, which means that $\eta$ is liftable with respect to $q$.
Therefore
\begin{equation}\label{inclbasica}
\Theta_Y\subseteq \Low_Z(q).
\end{equation}
An elementary computation shows that $\Theta_{Z}=(R^{-1})^*(R(\Theta_X))$, where $R(\Theta_X)=\{R(\eta): \eta\in \Theta_X\}$.
Hence $\Low_Z(q)=\Low_X(p)$ and thus (\ref{inclbasica}) implies that $\Theta_Y\subseteq \Low_X(p)$.

By hypothesis, the map $h\circ p:(\C^i,0)\to (\C^m,0)$ is an \textsc{icis} with $(h\circ p)^{-1}(0)$ of dimension $i-m\geq 1$. Then $h\circ p$ is reduced
(see \cite[p.\,7]{Looijenga}). Thus, as a direct application of Proposition \ref{inclLow}, the reverse inclusion $\Theta_Y\supseteq \Low_X(p)$ follows. Therefore $\Theta_Y=\Low_X(p)$.
\end{proof}


\begin{rem}
We have found that equality (\ref{LowandDer}) holds in a wide variety of examples where $X$ has not an isolated singularity at the origin.
We conjecture that Proposition \ref{LowX} holds at least when $X$ is homogeneous, not necessarily an \textsc{icis} with isolated singularity at the origin.
In particular, when $X$ is a generic determinantal variety.
\end{rem}

\section{Bruce-Roberts numbers and linear sections}\label{linearsects}

Let us fix a function $f\in\O_n$ and a complex analytic subvariety $X\subseteq (\C^n,0)$.
If $i\in\{1,\dots, n\}$, then we denote by $L_{i,n}(f,X)$ the set of those $p\in L_{i,n}$ such that
$\mu_{p^{-1}(X)}(f\circ p)$ is finite. As is already known in the case $X=\C^n$, the set
$L_{i,n}(f,X)$ can be strictly contained in $L_{i,n}$ even if $\mu_X(f)$ is finite.



Let us suppose that $f$ has an isolated singularity at the origin and let $i\in\{1,\dots, n\}$.
In \cite[p.\,299]{Cargese} Teissier showed that there exists a dense Zariski open set $U_{i,n}$ of the Grassmannian variety of
linear subspaces of dimension $i$ of $\C^n$ such that
the topological type of $f^{-1}(0)\cap H$ does not depend on $H$ whenever $H\in U_{i,n}$.
This leads to the definition of $\mu^{(i)}(f)$ as the Milnor number of the restriction
$f\vert_H$, where $H$ varies in $U_{i,n}$.
Moreover, due to the semicontinuity of the colength of ideals, the minimium possible value of the colength of the ideal
$J(f\circ p)=\langle \frac{\partial (f\circ p)}{\partial x_1},\dots,\frac{\partial (f\circ p)}{\partial x_i}\rangle$, where $p$ varies in $L_{i,n}(f,\C^n)$,
is actually equal to $\mu^{(i)}(f)$.
Motivated by this version of $\mu^{(i)}(f)$ we introduce in Definition \ref{defnustarX} the analogous concept in the context of Bruce-Roberts' Milnor numbers.

\begin{lem}\label{muXmin}
Let $f\in\O_n$ and let $X$ be an analytic subvariety of $(\C^n,0)$. Let $i\in\{1,\dots, n\}$ and let $p\in L_{i,n}(f,X)$. Then
$$
\mu_{p^{-1}(X)}(f\circ p)\geq \mu^{(i)}(f).
$$
\end{lem}

\begin{proof}
The inclusion $J_{p^{-1}(X)}(f\circ p)\subseteq J(f\circ p)$ is obvious, by the definition of $J_{p^{-1}(X)}(f\circ p)$.
The condition $p\in L_{i,n}(f,X)$ means that $J_{p^{-1}(X)}(f)$ has finite colength. Therefore $\mu(f\circ p)$ is finite and
thus $\mu_{p^{-1}(X)}(f\circ p)\geq \mu(f\circ p)\geq \mu^{(i)}(f)$.
\end{proof}




\begin{defn}\label{defnustarX}
Let $f\in\O_n$ and let $X$ be an analytic subvariety of $(\C^n,0)$.
For any $i\in\{1,\dots, n\}$ such that $L_{i,n}(f,X)\neq\emptyset$, we define the number
$$
\mu_X^{(i)}(f)=\min_{p\in L_{i,n}(f,X)}\mu_{p^{-1}(X)}(f\circ p).
$$
\end{defn}

If $L_{i,n}(f,X)=\emptyset$, then we set $\mu_X^{(i)}(f)=\infty$.
We denote the vector $(\mu_X^{(n)}(f),\dots, \mu_X^{(1)}(f))$ by $\mu_X^*(f)$.
We refer to $\mu_X^*(f)$ as the vector of {\it mixed Bruce-Roberts numbers of $f$ with respect to $X$}.

If $f\in\O_n$, $f\neq 0$, then the {\it order} of $f$ is defined as
$\ord(f)=\max\{r\in\Z_{\geq 1}: f\in\m_n^r\}$.
The order $\ord(I)$ of a non-zero ideal $I$ of $\O_n$ is defined analogously.

\begin{prop}
Let $X$ be an analytic subvariety of $(\C^n,0)$ with $\dim(X)<n$, $n\geq 2$. Let $f\in\O_n$, $f\neq 0$. Then $\mu_X^{(1)}(f)=\ord(f)$.
Consequently, if $\mu_X(f)<\infty$ and $\ord(f)\geq 3$, then
$$
\mu_X(f)\geq \mu_X^{(1)}(f).
$$
\end{prop}

\begin{proof}
Since $\dim(X)<n$, the intersection of $X$ with a generic line passing through the origin is equal to $\{0\}$.
Let $p\in L_{1,n}$ such that $p^{-1}(X)=\{0\}$. Let $Y=\{0\}\subseteq (\C,0)$.

Let us write $p$ as $p(x)=(x, a_2x, \dots, a_nx)$, for some $a_2,\dots, a_n\in\C$, for all $x\in \C$.
Let us take coordinates $x_1,\dots, x_n$ in $\C^n$.
Since $\Theta_Y=\m_1$, we have
$$
J_Y(f\circ p)=\left\langle x\frac{\partial (f\circ p)}{\partial x}\right\rangle=
\left\langle  x\frac{\partial f}{\partial x_1}(p(x))+a_2x\frac{\partial f}{\partial x_2}(p(x))+
\cdots+a_n x\frac{\partial f}{\partial x_n}(p(x))     \right\rangle.
$$
Let $I(f)$ denote the ideal of $\O_n$ generated by
$x_1\frac{\partial f}{\partial x_1},\dots, x_n\frac{\partial f}{\partial x_n}$. We have
$J_Y(f\circ p)\subseteq p^*(I(f))$ and $\ord(J_Y(f\circ p))=\ord(p^*(I(f))=\ord(I(f))=\ord(f)$, for a generic choice of the coefficients
$a_2,\dots, a_n$. Then
$$
\mu_X^{(1)}(f)=\dim_\C\frac{\O_1}{J_Y(f\circ p)}= \dim_\C\frac{\O_1}{p^*(I(f))}=\ord(f).
$$

If additionally we assume that $\mu_X(f)<\infty$, then $\mu_X(f)\geq \mu(f)\geq (\ord(f)-1)^n$.
We finally have that $(\ord(f)-1)^n\geq \ord(f)$, since we are assuming that $\ord(f)\geq 3$ and $n\geq 2$.
\end{proof}

The following example shows that the sequence $\mu^*_X(f)$ is not decreasing in general.

\begin{ex}
Let $f\in\O_3$ be the function given by $f(x,y,z)=x+y+z$ and let $X=\{(x,y,z)\in\C^3: xyz=0\}$.
We have $\Theta_{X}=\left\langle (x,0,0),(0,y,0),(0,0,z)\right\rangle$.
Therefore $\mu_X(f)=1$.

Let $p\in L_{2,3}$ be given by
$p(x,y)=(x,y,ax+by)$, where $a,b\in \C\setminus\{-1,0\}$. Therefore
$$
p^{-1}(X)=\left\{(x,y)\in\C^2: xy(ax+by)=0\right\}.
$$
By Theorem \ref{Derloghom}, we have that
$
\Theta_{p^{-1}(X)}=\left\langle (x,y),(ax^2+2bxy, -2axy-by^2)\right\rangle.
$
Thus
$$
J_{p^{-1}(X)}(f\circ p)=\left\langle x(a+1)+y(b+1), a(a+1)x^2+2(b-a)xy-b(b+1)y^2\right\rangle\subseteq\O_2.
$$
This implies that
$$
\mu_{X}^{(2)}(f)=\mu_{p^{-1}(X)}(f\circ p)=\dim_\C\frac{\O_2}{J_{p^{-1}(X)}(f\circ p)}=2.
$$
It is immmediate to check that $\mu_{X}^{(1)}(f)=1$. So $\mu_{X}^{*}(f)=(1,2,1)$.
\end{ex}

\begin{ex}
Let us consider the function $h:(\C^4,0)\to (\C,0)$ given by $h(x,y,z,t)=x^a+y^a+z^a+t^a$, for some $a\in\Z_{\geq 2}$. Let $X=h^{-1}(0)$.
Let $f\in\O_4$ be given by $f(x,y,z,t)=\alpha x^b+\beta y^b+\gamma z^b+\delta t^b$, where $b\in\Z_{\geq 2}$,
and $\alpha, \beta,\gamma,\delta$ denote generic complex coefficients. Therefore, we can apply \cite[Corollary 3.12]{NOT2011}
to deduce that 
\begin{align*}
\mu_X(f)&= b^4+(a-4)b^3+(a^2-4a+6)b^2+(a^3-4a^2+6a-4)b\\
\mu^{(3)}_X(f)&= b^3+(a-3)b^2+(a^2-3a+3)b \\
\mu^{(2)}_X(f)&= b^2+(a-2)b\\
\mu^{(1)}_X(f)&= b.
\end{align*}
\end{ex}

If $p\leq n$, given an integer $i\in\{1,\dots, n-p+1\}$, we denote by $\mu^{(i)}(g)$ the Milnor number of the
$\icis$ given by $(g_1,\dots, g_p, h_1,\dots, h_{n-p-i+1}):(\C^n,0)\to (\C^{n-i+1},0)$, where $h_1,\dots, h_{n-p-i+1}$ is a family of generic
linear forms of $\C[x_1,\dots, x_n]$ (see \cite{GH} or \cite{OSY}). Then $\mu^{(n-p+1)}(g)=\mu(g)$. Let us set $\mu^{(0)}(g)=1$. Hence,
as in the case $p=1$ (see \cite[p.\,300]{Cargese}), we also have a decreasing sequence of integers
$$
\mu^{(n-p+1)}(g)\geq \mu^{(n-p)}(g)\geq \cdots \geq \mu^{(1)}(g)\geq \mu^{(0)}(g).
$$
We will denote the vector $(\mu^{(n-p+1)}(g), \dots, \mu^{(1)}(g), \mu^{(0)}(g))$ by $\mu^*(g)$
and we refer to it as the {\it $\mu^*$-sequence of $g$}. Let us remark that, by (\ref{casequidim}),
we have
$$
\mu^{(1)}(g)=\dim_\C\frac{\O_n}{\langle g_1,\dots, g_p,h_1,\dots, h_{n-p}\rangle}-1,
$$
where $h_1,\dots, h_{n-p}$ is a family of generic linear forms of $\C[x_1,\dots, x_n]$.

Let $g:(\C^n,0)\to (\C^p,0)$ be an isolated complete intersection singularity.
We recall that, if $n-p\geq 1$, then the ring $\O_n/ \langle g_1,\dots, g_p\rangle$ is reduced (see \cite[p.\,7]{Looijenga}).
Following \cite[p.\,215]{Gaffney96}, we denote by $JM(g)$ the submodule of $(\O_n/\langle g_1,\dots, g_p\rangle)^p$ generated by the partial derivatives
$\frac{\partial g}{\partial x_1},\dots, \frac{\partial g}{\partial x_n}$.
Given a module of finite colength $M$ of a free module $R^p$, where $R$ denotes a given Noetherian local ring, then we denote by $e(M)$ the
Buchsbaum-Rim multiplicity of $M$.



\begin{prop}\label{split}
Let $h\in \C[x_1,\dots, x_n]$ be a homogeneous polynomial with isolated singularity at the origin and let $X=h^{-1}(0)$, $n\geq 2$.
Let $f\in\O_n$ such that $\mu_X(f)<\infty$. Then, for all $i\in\{2,\dots, n\}$:
\begin{equation}\label{nuis2}
\mu^{(i)}_X(f)=\mu^{(i)}(f)+\mu^{(i-1)}(f,h).
\end{equation}
Moreover, we have
\begin{equation}\label{sumademuisX}
\mu_X(f)+\mu_X^{(n-1)}(f)=
e\left(J(f)\frac{\O_n}{\langle f\rangle}\right)+e\big(JM(f,h)\big).
\end{equation}
\end{prop}

\begin{proof}
Let us fix an index $i\in\{2,\dots, n\}$. For a general $p\in L_{i,n}$, we have that $h\circ p:(\C^i,0)\to (\C,0)$
is also homogeneous with an isolated singularity at the origin. By Theorem \ref{claucashomog}, we have
\begin{equation}\label{splitting}
\mu_{p^{-1}(X)}(f\circ p)=\mu(f\circ p)+\mu(f\circ p,h\circ p).
\end{equation}
Let $p_{i+1},\dots, p_n$ denote the last $n-i$ components of $p$.
The Milnor number of the map $(f\circ p,h\circ p):(\C^i,0)\to (\C^2,0)$ is equal to the
Milnor number of $(f, h, x_{i+1}-p_{i+1}, \dots, x_{n}-p_{n}):(\C^n,0)\to (\C^2\times \C^{n-i},0)$, which
in turn is equal to $\mu^{(i-1)}(f,h)$, by the definition of the sequence of mixed Milnor numbers of an isolated complete intersection
singularity. Then (\ref{splitting}) shows relation (\ref{nuis2}).

By \cite[Corollaire 1.5]{Cargese} we know that $\mu(f)+\mu^{(n-1)}(f)=e(J(f)\frac{\O_n}{\langle f\rangle})$. Moreover, by
the Lê-Greuel formula and the definition of the sequence $\mu^*(f,h)$, for a generic choice of a linear form $\ell_1\in\C[x_1,\dots, x_n]$, we have
that
$$
\mu^{(n-1)}(f,h)+\mu^{(n-2)}(f,h)=\mu(f,h)+\mu(f,h, \ell_1)=\dim_\C\frac{\O_n}{\langle f,h\rangle+\J(f,h,\ell_1)}.
$$
This last colength is equal to $e\big(JM(f,h)\big)$, by
\cite[Proposition 2.6]{Gaffney96}. Then, by using (\ref{nuis2}) in the case $i=n$, we obtain that
\begin{align*}
\mu_X(f)+\mu_X^{(n-1)}(f)&=\mu^{(n)}(f)+\mu^{(n-1)}(f,h)+\mu^{(n-1)}(f)+\mu^{(n-2)}(f,h)\\
&=e\left(J(f)\frac{\O_n}{\langle f\rangle}\right)+e\big(JM(f,h)\big).
\end{align*}
\end{proof}

{\bf Acknowledgement.}
Part of this work was developed during the stay of the first author at the
Departmento de Matemática of ICMC, Universidade de São Paulo at São Carlos (Brazil), in February and July 2018.
The first author wishes to thank this institution for hospitality and working conditions and to FAPESP for financial support.



\end{document}